\documentclass[10pt]{article}
\date{August 29, 2011}
\usepackage{amssymb}
\usepackage{amsmath}
\input{liemacs10.sty}

\renewcommand{\mlabel}{\label}

\begin{document} 



\title{Oscillator algebras with semi-equicontinuous coadjoint 
orbits} 
\author{Karl-Hermann Neeb, Christoph Zellner
\begin{footnote}{Supported by DFG-grant NE 413/7-1, Schwerpunktprogramm 
``Darstellungstheorie''.}\end{footnote} 
}

\maketitle
\date 

\begin{abstract} A unitary representation 
of a, possibly infinite dimensional, Lie group 
$G$ is called semi-bounded if the corresponding 
operators $i\dd\pi(x)$ from the derived representations 
are uniformly bounded from above on some non-empty open subset 
of the Lie algebra $\g$. Not every Lie group has non-trivial semibounded unitary 
representations, so that it becomes an important issue to decide 
when this is the case. In the present paper we describe a complete 
solution of this problem for the class of generalized oscillator 
groups, which are semidirect products of Heisenberg groups with a 
one-parameter group~$\gamma$. For these groups it turns out that the existence 
of non-trivial semibounded representations is equivalent to the existence 
of so-called semi-equicontinuous non-trivial coadjoint orbits, a purely 
geometric condition on the coadjoint action. This in turn can be 
expressed by a positivity condition on the Hamiltonian function 
corresponding to the infinitesimal generator $D$ of~$\gamma$.
A central point of our investigations is that we make no assumption 
on the structure of the spectrum of~$D$. In particular, $D$ can be any 
skew-adjoint operator on a Hilbert space. 

\end{abstract} 

\tableofcontents 

\section*{Introduction}  \label{sec:0}

This paper is part of a project concerned with a 
systematic approach to unitary representations 
of infinite dimensional Lie groups in terms of conditions 
on spectra in the derived representations. For the derived 
representation to carry significant information, we first have to 
impose a suitable smoothness condition. A unitary 
representation $\pi \: G \to \U(\cH)$ 
is said to be {\it smooth} if the subspace 
$\cH^\infty$ of smooth vectors is dense. 
This is automatic for continuous 
representations of finite-dimensional groups. For any smooth 
unitary representation, the {\it derived representation} 
\[ \dd\pi \: \g = \L(G)\to \End(\cH^\infty), \quad 
\dd\pi(x)v := \derat0 \pi(\exp tx)v\]  
carries significant information in the sense that the closure of the 
operator $\dd\pi(x)$ coincides with the infinitesimal generator of the 
unitary one-parameter group $\pi(\exp tx)$. We call $(\pi, \cH)$ 
{\it semibounded} if the function 
\[ s_\pi \: \g \to \R \cup \{ \infty\}, \quad 
s_\pi(x) 
:= \sup\big(\Spec(i\dd\pi(x))\big) \]
is bounded on a neighborhood of some point in $\g$.  
All finite dimensional continuous unitary representations are bounded 
and most of the unitary representations appearing in physics are semibounded 
(\cite{BSZ92}, \cite{Bo96}, \cite{Ne10a}, \cite{Se67}, \cite{SeG81}, \cite{FH05}, \cite{GJ}). 
Therefore it is an important structural problem in infinite dimensional 
Lie theory to understand when a given Lie group $G$ 
has non-trivial semibounded unitary representations. In the 
present paper we obtain a complete answer to this problem for the 
class of {\it oscillator groups}.  Before we take a closer look 
at these groups, we describe our approach to this problem. 

A central idea is that, for a smooth unitary representation 
$(\pi, \cH)$, the function 
$s_\pi$ can be written as 
\[ s_\pi(x) 
= \sup \la I_\pi, -x \ra = - \inf \la I_\pi, x \ra, \] 
where $I_\pi \subeq \g'$ ($\g'$ denotes the topological dual of $\g$), 
is a weak-$*$-closed convex subset, called the {\it momentum set of 
$\pi$}. It is defined as the weak-$*$-closed convex hull of the image of the 
momentum map on the projective space of $\cH^\infty$: 
$$ \Phi_\pi \: \bP({\cal H}^\infty)\to \g' \quad \hbox{ with } \quad 
\Phi_\pi([v])(x) 
= \frac{1}{i}  \frac{\la  \dd\pi(x).v, v \ra}{\la v, v \ra}\quad \mbox{ for } 
[v] = \C v. $$
We call a subset $E \subeq \g'$ {\it semi-equicontinuous} 
if its {\it support functional} $s_E(x) := \sup \la E, -x \ra$ 
is bounded on the neighborhood of some point in $\g$. 
In this sense, a smooth unitary representation 
$(\pi, \cH)$ is semibounded if and only if $I_\pi$ is 
semi-equicontinuous. Since the ideal $I_\pi^\bot \trile \g$ 
is the kernel of the derived representation, we are interested 
in those groups for which the semi-equicontinuous coadjoint orbits 
$\cO_\lambda = \Ad^*(G)\lambda \subeq \g'$ separate the points of $\g$. 
This is trivially the case if $G$ is connected and abelian, 
which implies that all coadjoint orbits are {\it trivial}, i.e., 
$\cO_\lambda = \{\lambda\}$. Since the semibounded representations 
of abelian Lie groups are well-understood in terms 
of a suitable spectral theorem (\cite[Thm.~4.1]{Ne09}), 
it is natural to ask for the ``simplest'' class of non-abelian 
groups with non-trivial semibounded representations. 

As a consequence of the elementary  considerations in Section~\ref{sec:1}, 
we show in Theorem~\ref{thm:1.3} that if $\g$ is either 
nilpotent or $2$-step solvable, then all semi-equicontinuous coadjoint 
orbits in $\g'$ are trivial. On the group level, this implies 
that all semibounded unitary representations factor through abelian 
quotient groups. Therefore we focus on $3$-step solvable 
Lie algebras. Among these algebras, the simplest and the most 
prominent class are the oscillator algebras. They are constructed 
as follows. 
Let $(V,\omega)$ be a locally convex symplectic space and 
$$ \Heis(V,\omega) := \R \oplus_\omega V $$
be the corresponding Heisenberg group with the multiplication 
$$ (z,v)(z',v') = (z + z' + \shalf \omega(v,v'), v + v'). $$
Further, 
let $\gamma \: \R \to \Sp(V,\omega)$ be a representation defining a smooth 
action of $\R$ on $V$ and $D := \gamma'(0)$ be its infinitesimal 
generator. Then 
$$ G := G(V,\omega,\gamma) := \Heis(V,\omega) \rtimes_\gamma \R $$
is a Lie group, called 
the corresponding (generalized) {\it oscillator group}. 
Its multiplication is given by 
$$ (z,v,t)(z',v',t') = \Big(z + z' + \frac{1}{2}\omega(v,\gamma(t)v'), 
v + \gamma(t)v', t + t'\Big)  $$
and its Lie algebra is 
$$ \g := \g(V,\omega,D) := \heis(V,\omega) \rtimes_D \R $$  
with the bracket 
$$ [(z,v,t), (z',v',t')]= (\omega(v,v'), tDv' - t'Dv,0). $$
The one-parameter group $\gamma$ represents on $V$ the flow 
of the Hamiltonian vector field corresponding to the 
function 
$H_D(v) := \shalf \omega(Dv,v).$

The first main result of the present paper 
is a characterization of those oscillator algebras 
with non-trivial semi-equicontinuous 
coadjoint orbits. The answer is surprisingly 
simple, namely that such orbits exist if and only if 
$\g(V,\omega, D)$ is isomorphic to a double extension 
of a euclidean locally convex space $(V,\kappa)$, defined by a 
skew-symmetric derivation~$D$ on~$V$  (Theorem~\ref{thm:4.2}).
 We also give a description of the set 
$\g'_{\rm seq}$ of semi-equicontinuous coadjoint orbits 
(Proposition~\ref{prop:char2}). In particular, we show that the 
existence of non-trivial semi-equi\-continuous 
coadjoint orbits implies that 
$H_D$ is either positive or negative. 

In Sections~\ref{sec:5} and \ref{sec:6} we get full circle by showing 
that, up to passing to a natural completion, resp., a naturality 
condition on the topology on $V$ in terms of $\kappa$ and $D$, 
the existence of semi-equicontinuous coadjoint orbits 
is for oscillator groups equivalent to the existence of 
non-trivial semibounded unitary representations. 
The investigation of the structure of the set of 
semibounded unitary representations of a given oscillator 
group is part of the second author's Ph.D.\ project (cf. \cite{Ze11}). 

The investigations in the present paper are prototypical for many 
other situations occurring in infinite dimensional Lie theory in relation to 
mathematical physics. There one often encounters groups of the from 
$G = \hat H \rtimes_\gamma \R$, where $\hat H$ is a central extension of a Lie 
group~$H$. For such groups the question arises if there are unitary 
representations $(\pi, \cH)$ satisfying the ``positive energy'' condition, 
also called the ``spectral condition'', that $-i\dd\pi(0,1)$ is bounded 
from below (\cite{Bo96}). 
Generalized oscillator groups are the simplest groups of this 
type, corresponding to the situation where $H = (V,+)$ is the additive 
group of a locally convex space. In this case the positive energy 
condition is equivalent to semiboundedness (cf.\ Lemma~\ref{lem:semiboundedcond}). 
The necessary conditions derived in the present paper apply 
in particular to $\gamma$-invariant abelian subgroups of $H$ so that they 
have an immediate impact on the general case. In view of the absence of 
any structure of the group $H = (V,+)$, they provide the context 
with the maximal freedom for the one-parameter groups~$\gamma$ so that 
the influence of the spectral properties of $\gamma$ on the representation 
theory of $G$ can be studied most directly.

\section{Triviality of semi-equicontinuous orbits} \mlabel{sec:1}

In this section we introduce the main geometric concept of this paper: 
semi-equicontinuous coadjoint orbits. 
The main result of this short section is Theorem~\ref{thm:1.3}, 
asserting that if $\g$ is nilpotent or $2$-step solvable, then 
all its semi-equicontinuous coadjoint orbits are trivial. 
Throughout $\g$ denotes the Lie algebra of a connected 
locally convex Lie group~$G$ to ensure the existence of a coadjoint 
action $\Ad^*(g)\alpha := \alpha \circ \Ad(g)^{-1}$.

\begin{defn} We call a subset $E \subeq \g'$ {\it semi-equicontinuous} 
if its {\it support functional} $s_E(x) := \sup \la E, -x \ra$ 
is bounded on the neighborhood of some point in $\g$. 
This implies in particular, that the associated convex cone 
\[ B(E) := \{ x \in \g \: \inf \la E, x \ra > - \infty \}  
= \{ x \in \g \: s_E(x) < \infty \} \]  
has interior points. 
\end{defn}

\begin{rem} \mlabel{rem:2.2} 
(a) If $E$ is semi-equicontinuous, then $B(E)$ has interior points 
and $E$ contains no affine line $\alpha_0 + \R \beta$, because 
this would lead to 
\[ B(E) \subeq B(\alpha_0 + \R \beta) \subeq \ker \beta.\]

(b) If $\g$ is finite dimensional, then a subset $E \subeq \g'$ 
is semi-equicontinuous if and only if its convex hull contains 
no affine lines (cf.\ \cite[Prop.~V.1.15]{Ne00}). 

(c) If $W \subeq \g$ is an open convex cone, then its dual cone 
\[ W^\star := \{ \alpha \in \g' \: \alpha(W) \subeq [0,\infty[\} \] 
is semi-equicontinuous because 
$s_{W^\star}(W) = \{0\}$. 
\end{rem}

If $\g'_{\rm seq} \subeq \g'$ denotes the set of all semi-equicontinuous 
coadjoint 
orbits, then its annihilator $\fn := (\g'_{\rm seq})^\bot$ is a closed 
$\Ad(G)$-invariant ideal in $\g$, and each semi-equicontinuous coadjoint orbit 
can be identified with a coadjoint orbit of the quotient algebra 
$\g/\fn$. 

\begin{lem} \mlabel{lem:ideal} 
Suppose that 
$\cO_\lambda \subeq \g'$ is semi-equicontinuous. 
If $x \in \g$ satisfies $(\ad x)^2 = 0$, then 
$[x,\g] \subeq \cO_\lambda^\bot$. 
In particular, we have: 
\begin{description}
\item[\rm(a)]  Every abelian ideal $\fa \trile \g$ satisfies 
$[\fa,\g] \subeq \cO_\lambda^\bot$. 
\item[\rm(b)]  If $\g'_{\rm seq}$ separates the points of $\g$, 
i.e., $(\g'_{\rm seq})^\bot=\{0\}$, 
then every abelian ideal in $\g$ is central.
\end{description}
\end{lem}

\begin{prf} From 
$\cO_\lambda \supeq \lambda \circ e^{\R \ad x} 
= \lambda + \R (\lambda \circ \ad x)$ 
and the semi-equicontinuity of $\cO_\lambda$,   
we derive that $\lambda \circ \ad x = 0$. 
Applying the same argument to any other 
$\mu \in \cO_\lambda$ implies that 
$[x,\g] \subeq \cO_\lambda^\bot$. 

(a) If $\fa$ is an abelian ideal, then $(\ad x)^2 = 0$ holds 
for every $x \in \fa$, so that (a) follows from the preceding 
argument. 

(b) is an immediate consequence of (a). 
\end{prf}

For a Lie algebra $\g$, we define the {\it derived series} 
$$ D^0(\g) := \g, \quad D^{n+1}(\g) := [D^n(\g), D^n(\g)] 
\quad \mbox{ for } \quad n \in \N_0, $$
and the {\it descending central series} 
$$ C^1(\g) := \g, \quad C^{n+1}(\g) := [\g, C^n(\g)] 
\quad \mbox{ for } \quad n \in \N.$$

\begin{lem}\mlabel{lem:2.2} 
If $\g$ is a Lie algebra in which all abelian ideals are central, 
then the following assertions hold: 
\begin{description}
\item[\rm(a)] Each nilpotent ideal 
$\fn \trile \g$ satisfies $C^2(\fn) \subeq \z(\g)$ and $C^3(\fn) = \{0\}$. 
\item[\rm(b)] If $\fr \trile \g$ is a solvable ideal and 
$n \in \N_0$ is maximal with $D^n(\fr) \not=\{0\}$, then 
$D^n(\fr)$ is central in $\g$. 
\item[\rm(c)] $\z(\ad \g) =\{0\}$. 
\item[\rm(d)] If $\g$ is nilpotent, then $\g$ is abelian. 
\end{description}
\end{lem} 

\begin{prf} (a) Let $n \in \N_0$ be such that 
$C^{n+1}(\fn)=\{0\}$ and minimal with this property. 
We may w.l.o.g.\ assume that 
$\fn$ is non-abelian, i.e., $n \geq 2$. 
If $n > 2$, then the relation 
\[  [C^{n-1}(\fn), C^{n-1}(\fn)] 
\subeq C^{2n-2}(\fn) \subeq C^{n+1}(\fn)= \{0\} \] 
leads to $C^{2n-2}(\fn)= \{0\}$, so that $C^{n-1}(\fn)$ is an abelian ideal 
of $\g$, hence central. This contradicts $C^n(\fn) \not=\{0\}$. 
Therefore $n \leq 2$, and this means that $C^3(\fn) = \{0\}$. 
This in turn implies that 
$[C^2(\fn), C^2(\fn)] \subeq C^4(\fn) = \{0\}$, so that $C^2(\fn)$ 
is an abelian ideal of $\g$, hence central. 

(b) This is an immediate consequence of the fact that 
$D^n(\fr)$ is an abelian ideal of $\g$. 

(c) Suppose that $\ad x \in \z(\ad \g)$. Then 
$\fa := \ad^{-1}(\R \ad x) = \R x + \z(\g)$ is an abelian 
ideal of $\g$, hence central. We conclude that 
$x \in \z(\g)$, and finally that $\ad x = 0$. 

(d) From (a) we derive that $C^2(\g) = [\g,\g]$ is central in $\g$, 
and thus $\ad \g$ is abelian.  In view of (c), this leads to 
$\ad \g = \{0\}$, i.e., $\g = \z(\g)$ is abelian. 
\end{prf}

\begin{thm} \mlabel{thm:1.3} If $\g$ is either nilpotent or 
two step solvable, i.e., $D^2(\g) = \{0\}$,  
then $\g'_{\rm seq} \subeq [\g,\g]^\bot$, i.e., 
all semi-equicontinuous coadjoint orbits are trivial, i.e., 
consist of one point. 
\end{thm} 

\begin{prf} Passing to the quotient $\fq := \g/\fn$, 
where $\fn := (\g'_{\rm seq})^\bot$, we obtain a 
Lie algebra with the same semi-equicontinuous coadjoint 
orbits as $\g$, and for which $\fq'_{\rm seq}$ separates the 
points. Hence every abelian ideal in $\fq$ is central 
by Lemma~\ref{lem:ideal}. 

If $\g$ is two step solvable, then the same holds for 
$\fq$, and Lemma~\ref{lem:2.2} implies that the abelian ideal 
$[\fq,\fq] \trile \fq$ is central, i.e., $C^3(\fq) = \{0\}$. 
In particular, $\fq$ is nilpotent. 
Finally, Lemma~\ref{lem:2.2}(d) implies that $\fq$ is abelian, 
i.e., that $[\g,\g] \subeq \fn$. This means that 
every semi-equicontinuous coadjoint orbit in $\g$ is trivial. 
\end{prf}

\section{Oscillator algebras with semi-equicontinuous coadjoint orbits} 
\mlabel{sec:3}

The two main results of this section are the 
characterization of those oscillator algebras 
$\g(V,\omega,D)$ with non-trivial semi-equicontinuous coadjoint orbits 
(Theorem~\ref{thm:char}) and the description of the 
semi-equicontinuous coadjoint orbits for these algebras 
(Proposition~\ref{prop:char}). 

\begin{defn} A pair $(V,Q)$, where $V$ is a locally convex 
space and $Q$ is a continuous symmetric positive definite 
bilinear form on $V$ is called a {\it locally convex euclidean space}. 
For $\alpha \in V'$ we then write 
\[ \|\alpha\|_Q := \sup \{ \alpha(v) \: Q(v,v) \leq 1 \}
\in [0,\infty] \] 
for the norm of $\alpha$ with respect to~$Q$.
We write $V_Q$ for the completion of $V$ with respect to $Q$ and 
observe that the condition $\|\alpha\|_Q < \infty$ is equivalent 
to the existence of a continuous extension of $\alpha$ to~$V_Q$.
\end{defn}

\begin{lem}\mlabel{lem:quad} 
Let $(V,Q)$ be a locally convex euclidean space. 
Then the quadratic polynomial 
$$ f_\alpha \: V \to \R, \quad f_\alpha(v) = \alpha(v) + Q(v,v) $$
is bounded from below if and only if 
$Q$ is positive semidefinite and $\alpha$ is $Q$-continuous, i.e., 
$\|\alpha\|_Q < \infty.$
In this case 
$$ \inf f_\alpha = - \frac{1}{4} \|\alpha\|_Q^2. $$
In particular, for a subset $A \subeq V'$, the set 
$\{ f_\alpha \: \alpha \in A \}$ 
is uniformly bounded from below on $V$ if and only 
if 
$$\sup \{ \|\alpha\|_Q \: \alpha \in A \} < \infty. $$
\end{lem} 

\begin{prf} If $Q(v,v) < 0$, then $f_\alpha$ is not bounded from below 
on the line 
$\R v$. Therefore $Q$ is positive semidefinite if $f_\alpha$ is bounded 
from below. It is also necessary that $\alpha$ vanishes on each 
$Q$-isotropic vector. 
If $\alpha$ is unbounded on the $Q$-unit ball 
$B_Q := \{ v \in V \: Q(v,v)< 1\}$, then the convexity of $B_Q = - B_Q$ 
implies that 
$f_\alpha(B_Q) = \R$. Therefore 
$\|\alpha\|_Q < \infty$ is necessary for $f_\alpha$ to be bounded from below. 

Suppose that $\inf f_\alpha > - \infty$. Then 
$\alpha$ is represented by an element $v_\alpha \in V_Q$, and 
$$\oline f_\alpha(v) := \alpha(v) + Q(v,v) = Q(v_\alpha, v) + Q(v,v) $$
defines a polynomial map $V_Q \to \R$. For $v_0 \in V_Q$ we have 
$$\dd \oline f_\alpha(v_0) = \alpha + 2 Q(v_0,\cdot) 
= Q(v_\alpha + 2 v_0,\cdot), $$
which vanishes only for $v_0 = -\frac{1}{2}v_\alpha$. In this point 
$\oline f_\alpha$ takes its minimal value 
$$ \inf \oline f_\alpha 
= - \frac{1}{2} \alpha(v_\alpha) + \frac{1}{4}Q(v_\alpha, v_\alpha)
= - \frac{1}{4}Q(v_\alpha, v_\alpha)
= - \frac{1}{4}\|\alpha\|_Q^2. $$
Since $V$ is dense in $V_Q$, we finally derive 
$\inf f_\alpha = \inf \oline f_\alpha = - \frac{1}{4}\|\alpha\|_Q^2.$
\end{prf}

\begin{rem}\mlabel{rem:3.2} For a quadratic functional of the form 
$f(v) = \alpha(v) + t Q(v,v)$ 
for which $Q$ is positive definite and $t > 0$, we obtain 
with Lemma~\ref{lem:quad} 
$$ \inf f = t\cdot \inf(t^{-1} \alpha + Q) 
= -\frac{t}{4} \|t^{-1}\alpha\|_Q^2 
= -\frac{1}{4t} \|\alpha\|_Q^2. $$
\end{rem}

\begin{rem} \mlabel{rem:switchd} Note that the map 
$$ \phi \: \heis(V,\omega) \rtimes_D \R \to \heis(V,\omega) \rtimes_{-D} \R, 
\quad (z,v,t) := (z,v,-t)$$ 
defines an isomorphism of Lie algebras. Therefore we can always 
pass from $D$ to $-D$ if this is convenient. An analogous 
statement holds for the corresponding Lie groups.
\end{rem}
  
We want to describe the subset 
$\g'_{\rm seq}$ of semi-equicontinuous coadjoint orbits 
and in particular find necessary 
and sufficient conditions on $(V,\omega, D)$ 
for the existence of {\it non-trivial} 
semi-equicontinuous coadjoint orbits. 

\begin{rem}
  \mlabel{rem:3.4} 
For $\lambda = (z^*, \alpha, t^*)$ and 
$v \in V$ we derive from $(\ad v)^3 = 0$ that 
\begin{align} \label{eq:acoad1}
\Ad^*(0,v,0)\lambda 
&= \lambda - \lambda \circ \ad v + \frac{1}{2} \lambda \circ (\ad v)^2\notag \\
&= \lambda + (0, -z^* i_v \omega,  \alpha(Dv)) 
 +  \frac{1}{2} (0,0, z^* \omega(Dv,v))\notag \\
&= \Big(z^*, \alpha - z^* i_v\omega, t^* + \alpha(Dv) 
+  \frac{z^*}{2} \omega(Dv,v)\Big), 
\end{align}
i.e., 
\begin{equation}
  \label{eq:acoad2}
(\Ad^*(0,v,0)\lambda)(z,x,t) = 
z^*z + \alpha(x) - z^* \omega(v,x) + tt^* 
+ t \alpha(Dv) + \frac{tz^*}{2} \omega(Dv,v). 
\end{equation}
In particular, 
$\Ad^*(V)\lambda = \{\lambda\}$ is equivalent to 
$z^* = 0$ and $D^*\alpha = 0$, which is equivalent to 
the triviality of the coadjoint orbit $\cO_\lambda$, resp., to 
$\lambda \in [\g,\g]^\bot$. 

Assume that $Q(v,w) := \omega(Dv,w)$ is positive semidefinite and 
that $tz^* > 0$, so that the quadratic term in \eqref{eq:acoad2} 
is positive definite on $V$. 
From Remark~\ref{rem:3.2} we then derive that 
\begin{equation}
  \label{eq:acoad3}
\inf\la \Ad^*(0,V,0)\lambda, (z,x,t) \ra
= z^*z + \alpha(x) + tt^* - \frac{1}{2tz^*} 
\|t D^*\alpha + z^* i_x\omega\|_Q^2. 
\end{equation}
\end{rem}

\begin{lem} \mlabel{lem:2.3} 
Suppose that $\lambda = (z^*,\alpha,t^*)\in \g'_{\rm seq}$ 
and that there exists an element 
$(z,x,t) \in B(\cO_\lambda)^0$ with $z^* t > 0$. 
Then the following assertions hold: 
\begin{description}
\item[\rm(i)] The symmetric bilinear form 
$Q \: V \times V \to \R, Q(v,w) := \omega(Dv,w)$
is positive definite.
\item[\rm(ii)] $D^*\alpha$ and all linear forms 
$i_x\omega$, $x \in V$, are $Q$-continuous with 
$x \mapsto \|i_x\omega\|_Q$ bounded on a $0$-neighborhood in $V$. 
\end{description}
\end{lem}

\begin{prf} (i) In view of \eqref{eq:acoad2}, the quadratic function 
$$ f_x(v) := (\Ad^*(v)\lambda)(z,x,t) = 
z^*z + \alpha(x) - z^* \omega(v,x) + tt^* 
+ t \alpha(Dv) + \frac{tz^*}{2} \omega(Dv,v) $$
on $V$ is bounded from below. Its 
restriction to $\R v$ is bounded from below if and only if 
either 
$\omega(Dv,v) >  0$ or 
\begin{equation}  \label{eq:flatcond}
\omega(Dv,v) = 0 \quad \mbox{ and } \quad t\alpha(Dv) = - z^*  \omega(x,v).
\end{equation}
This shows that $Q$ is positive semidefinite. 

To see that $Q$ is actually positive definite, suppose that 
$v \in V$ satisfies  $Q(v,v) = 0$. Then the positive semidefiniteness 
of $Q$ implies $\{0\} = Q(v,V) = \omega(Dv,V)$, which in turn leads to 
$Dv = 0$ because $\omega$ is non-degenerate. By \eqref{eq:flatcond} we now obtain that $\omega(x,v)=0$.
Since $(z,x,t) \in B(\cO_\lambda)^0$ and $\lambda \in \g'_{\rm seq}$, there exists an open $0$-neighborhood 
$U \subeq V$ such that 
$$ \inf_{x' \in x + U} \inf 
\la \cO_\lambda, (z,x',t) \ra > - \infty. $$
In particular, $(z,x',t) \in B(\cO_\lambda)^0$, 
so that the preceding argument leads to $\omega(U,v) = \{0\}$, 
and hence to $v = 0$. This proves that $Q$ is positive definite. 

(ii) Let $U$ be as in (i) above. 
From Lemma~\ref{lem:quad} and the uniform semiboundedness 
of the quadratic functions $f_{x'}$, $x' \in x + U$, from below, 
we derive 
that the set of linear functionals 
\[  z^* i_{x'}\omega + t D^* \alpha, \quad x' \in x + U, \] 
is uniformly bounded with respect to $\|\cdot\|_Q$. As $z^* \not=0$, 
the set 
$i_y\omega$, $y \in U$, is uniformly $Q$-bounded. 
This implies in particular that all functionals $i_v\omega$, $v \in V$, 
are $Q$-bounded, and hence that $D^*\alpha$ is also $Q$-bounded. 
\end{prf} 

\begin{rem} \mlabel{rem:orbred} If $\lambda = (z^*, \alpha, t^*)\in \g'$ 
vanishes on the central element $(1,0,0)$, i.e., $z^* = 0$, then 
$\cO_\lambda$ vanishes on $(1,0,0)$, hence 
can be considered as a coadjoint orbit of the 
$2$-step solvable quotient Lie algebra 
$V \rtimes_D \R$. Now Theorem~\ref{thm:1.3} implies 
that $\cO_\lambda$ is semi-equicontinuous if and only if 
it is trivial.   We conclude that $z^* \not=0$ whenever $\cO_\lambda$ 
is nontrivial and semi-equicontinuous. 
\end{rem}

\begin{thm} \mlabel{thm:char} {\rm(Characterization Theorem)}
The Lie algebra $\g = \g(V,\omega,D)$ has non-trivial semi-equicontinuous 
coadjoint orbits if and only if, possibly after replacing $D$ by 
$-D$, the following conditions are satisfied: 
\begin{description}
\item[\rm(C1)] $Q(x,y) := \omega(Dx,y)$ is positive definite. 
\item[\rm(C2)] All functionals $i_x\omega$, $x \in V$, are $Q$-continuous. 
\item[\rm(C3)] If $V_Q$ denotes the Hilbert completion of $V$ with respect to 
$Q$, then the map 
$\eta \: V \to V_Q$, defined by 
$$ Q(\eta(x), y) = (i_x\omega)(y) = \omega(x,y) \quad \mbox{ for } \quad 
x,y \in V, $$
is continuous. 
\end{description}
\end{thm} 

\begin{prf} Suppose that $\lambda = (z^*,\alpha,t^*) \in \g'$ 
has a non-trivial semi-equicontinuous orbit. 
In view of Remark~\ref{rem:orbred}, we have $z^*\not=0$. 

Since the open cone $B(\cO_\lambda)^0$ is non-empty, 
it either contains an element $(z,x,t)$ with 
$t > 0$ or $t< 0$. In the case $z^*t<0$ we replace $D$ by $-D$ 
(cf.\ Remark~\ref{rem:switchd}). Now (C1)-(C3) are consequences 
of Lemma~\ref{lem:2.3}. 

Suppose, conversely, that (C1)-(C3) are satisfied. We claim that 
$\lambda := (1,0,0)$ has a semi-equicontinuous orbit. 
Since $\lambda$ is fixed by the subgroup $\{(0,0)\} \times \R$, we have 
$\cO_\lambda = \Ad^*(0,V,0)\lambda$. Now 
\eqref{eq:acoad3} implies that, for $t > 0$,  
$$ \inf \cO_\lambda(z,x,t) 
=  \inf \la \Ad^*(0,V,0)\lambda, (z,x,t)\ra 
= z - \frac{1}{2t}\|i_x \omega\|_Q^2. $$ 
This is a continuous function on the open half space 
$\heis(V,\omega) \times ]0,\infty[$, so that 
$\cO_\lambda$ is semi-equicontinuous. 
\end{prf}

\begin{rem} \mlabel{rem:3.8} 
In Theorem~\ref{thm:char}, 
the positive definite form $Q$ is $\gamma$-invariant. As a consequence, 
the embedding $\eta \: V \to V_Q$ is $\gamma$-equivariant. 
\end{rem}

We now take a closer look at the set of semi-equicontinuous 
coadjoint orbits for generalized oscillator algebras 
satisfying the conditions (C1)-(C3) in 
Theorem~\ref{thm:char}. We shall see in Proposition~\ref{prop:char2} 
below how these conditions simplify from the perspective of double 
extensions.

\begin{prop} \mlabel{prop:char} 
If {\rm(C1)-(C3)} in {\rm Theorem~\ref{thm:char}} are satisfied, then 
the coadjoint orbit 
of $\lambda = (z^*, \alpha, t^*)$ is non-trivial and  
semi-equicontinuous if and only if 
\begin{description}
\item[\rm(SE1)]  $z^*\not=0$.
\item[\rm(SE2)] $D^*\alpha$ is $Q$-continuous. 
\item[\rm(SE3)] $\gamma(\R)^*\alpha\subeq V'$ is semi-equicontinuous.  
\end{description}
In this case we have 
\begin{equation}
  \label{eq:bouncon}
B(\cO_\lambda)^0 = \R \times B(\gamma(\R)^*\alpha)^0 \times ]0,\infty[z^*. 
\end{equation}
\end{prop}

\begin{prf} Suppose first that $\cO_\lambda$ is non-trivial 
and semi-equicontinuous. Then Remark~\ref{rem:orbred} implies that 
$z^*\not=0$. With \eqref{eq:acoad2} 
and Lemma~\ref{lem:quad} we see that 
$z^* t \geq 0$ is a necessary 
condition for $\Ad^*(0,V,0)(z,x,t)$ to be bounded from below. This implies the 
existence of an element $(z,x,t) \in B(\cO_\lambda)^0$ with 
$z^*t > 0$. Now Lemma~\ref{lem:2.3}(ii) implies that $D^*\alpha$ 
is $Q$-continuous. In view of (C2) and (C3), 
\eqref{eq:acoad3} now leads to 
$$ B(\Ad^*(0,V,0)\lambda)^0 = \heis(V,\omega) \times ]0,\infty[ z^*.$$ 

Next we note that 
$$ \cO_\lambda(z,x,t) 
= \la \Ad^*(0,V,0)\lambda, (z,\gamma(\R)x,t) \ra. $$
Since $\gamma(\R) \subeq \GL(V)$ preserves $\omega$ and $D$, it is 
$Q$-isometric. Therefore 
$$\|i_{\gamma(s)x}\omega\|_Q = \|i_x\omega\|_Q \quad \mbox{ for  each } 
s \in \R,$$ 
and this implies that 
$\|z^* i_{\gamma(s)x} \omega + t D^*\alpha\|_Q^2$ 
is uniformly bounded in $s \in \R$ and locally uniformly in $x$. 
Hence the rightmost term in \eqref{eq:acoad3} is uniformly 
bounded if $x$ runs through $\gamma(\R)U_x$ for some neighborhood 
$U_x$ of $x$. Therefore $(z,x,t) \in B(\cO_\lambda)^0$ 
implies that 
$$ y \mapsto \inf \alpha(\gamma(\R)y) $$
is bounded on some neighborhood of $x$, 
which means that $\gamma(\R)^*\alpha$ is semi-equicontinuous. 
This proves the necessity of (SE1)-(SE3). 

Suppose, conversely, that (SE1)-(SE3) are satisfied. 
Then (SE2) and (SE3) imply that in \eqref{eq:acoad3} 
the right hand side is uniformly bounded below on a 
$\gamma(\R)$-invariant neighborhood of $(z,x,t)$, and this means that 
$\inf \cO_\lambda(z,x,t)$ is bounded from below on a neighborhood 
of $(z,x,t)$, which shows that $\cO_\lambda$ is semi-equi\-con\-tinuous.
\end{prf}

\section{The connection with double extensions} 
\mlabel{sec:4}

Let $(V,\kappa)$ be a locally convex euclidean space and 
$D \: V \to V$ be a skew-symmetric endomorphism of $V$, $D\neq0$.
Then $\omega_D(x,y) := \kappa(Dx,y)$ defines a skew-symmetric 
bilinear form on $V$, so that we obtain a 
Lie algebra 
\[  \g(V,\kappa,D) := (\R \oplus_{\omega_D} V) \rtimes_D \R \] 
with the bracket 
\[ [(z,v,t), (z',v',t')] 
= (\omega_D(v,v'), tDv' - t'Dv,0)
= (\kappa(Dv,v'), tDv' - t'Dv,0).\] 
This Lie algebra $\g(V,\kappa, D)$ carries the continuous Lorentzian form 
\[ \kappa_\g((z,x,t),(z',x',t')) := zt' + z't + \kappa(x,x'), \] 
which is easily seen to be invariant under the adjoint action 
of the Lie algebra on itself, hence also under the adjoint action 
of any corresponding connected Lie group. 

\begin{rem} If $\g := \g(V,\kappa, D)$ is a double extension as above, then a slight modification 
of the proof of \cite[Thm.~VI.2.7]{Ne06} 
implies that the existence of a Lie group $G$ with 
Lie algebra $\g$ is equivalent to the integrability of 
$D$ to a smooth one-parameter group 
$\gamma_D$ on $V$. 
\end{rem}

From now on we assume that $D$ integrates to a smooth orthogonal one-parameter group $\gamma$ on $V$. Since $\kappa_\g$ is a continuous invariant Lorentzian form, it follows that
$$ W := \{ (z,x,t) \: z > 0, \kappa_\g((z,x,t), (z,x,t)) < 0\} $$ 
is an open invariant convex cone in $\g(V,\kappa,D)$, so that its  dual cone $W^\star$ is a semi-equicontinuous subset which 
contains for $V \not=\{0\}$ non-trivial coadjoint orbits, such as the orbit of 
the functional ${\lambda(z,x,t):= z}$ (Remark~\ref{rem:2.2}(c)). 

In view of this observation, the 
following characterization captures the essence of the 
existence of non-trivial semi-equicontinuous coadjoint orbits 
for oscillator algebras $\g(V,\omega,D)$. For us, this leads to a 
change of 
perspective with respect to the initial data from the symplectic 
data $(V,\omega, D)$ to the euclidean data $(V,\kappa,D)$, 
which is better adapted to the discussion of semi-equicontinuous 
coadjoint orbits.

\begin{thm} \mlabel{thm:4.2} A generalized oscillator algebra 
$\g(V,\omega, D)$ with $V \not=\{0\}$ 
contains non-trivial semi-equicontinuous coadjoint 
orbits if and only if it is isomorphic to a double extension 
$\g(V,\kappa,D)$ defined by a continuous skew-symmetric operator $D$ 
on a euclidean locally convex space $(V,\kappa)$. 
\end{thm}

\begin{prf} If $\g := \g(V,\omega, D)$ contains non-trivial 
semi-equicontinuous coadjoint 
orbits, then by Remark~\ref{rem:switchd} we may assume that (C1)-(C3) from Theorem~\ref{thm:char}  
are satisfied. Hence the embedding $\eta \: V \to V_Q$ leads to a 
$\gamma(\R)$-invariant continuous scalar product on~$V$: 
$$ \kappa(x,y) := Q(\eta(x), \eta(y)) $$  
satisfying 
$\kappa(x,x) = \|i_x\omega\|_Q^2 = \|\eta(x)\|_Q^2$ 
(Remark~\ref{rem:3.8}). 

We may consider $\eta$ as an unbounded skew-symmetric operator 
on the real Hilbert space $V_Q$ with domain 
$\cD = V$. It satisfies 
$$ Q(\eta(Dx),y) = \omega(Dx,y) = Q(x,y), \quad x,y \in V, $$
so that $\eta \circ D = \id_V$ and 
\begin{equation}
  \label{eq:kappa0}
\kappa(Dx,Dy) = Q(x,y) \quad \mbox{ for } \quad x,y \in V.   
\end{equation}
For $x,y \in V$ we further obtain 
\begin{equation}
  \label{eq:4.1}
\omega_D(x,y) := \kappa(Dx,y) = Q(\eta(Dx), \eta(y)) = Q(x, \eta(y)) 
= Q(\eta(y),x) = \omega(y,x). 
\end{equation}
This implies that the map 
$$ \g(V,\kappa,D) := (\R \oplus_{\omega_D} V) \rtimes_D \R \to 
\g(V,\omega, D), \quad 
(z,x,t) \mapsto (-z,x,t) $$
is an isomorphism of topological Lie algebras. 

If, conversely, $D$ is a skew-symmetric operator on the euclidean 
locally convex space, then we have already seen above that 
$\g(V,\kappa,D)$ has non-trivial coadjoint orbits. 
\end{prf}

\begin{rem} \mlabel{rem:4.3}
For a generalized oscillator algebra 
$\g(V,\omega, D)$, $V \not=\{0\}$ 
with non-trivial semi-equicontinuous coadjoint 
orbits we have constructed in the proof of the preceding proposition 
a continuous real scalar product $\kappa$ on $V$ which, possibly after replacing $D$ by $-D$, satisfies $\omega(x,y)=-\kappa(Dx,y)$ for $x,y\in V$. This $\kappa$ has the nice property that $D(V)$ is dense in $V_\kappa$. To see this choose for a $v\in V$ a sequence $v_n\in V$ such that $v_n$ converges to $\eta(v)\in V_Q$ with respect to the $Q$-norm. 
Since $v_n=\eta(D v_n)$, it follows that 
\[ \|D v_n - v\|_\kappa = \|\eta(Dv_n) - \eta(v)\|_Q \to 0.\] 
\end{rem}

From now on we suppose $\g(V,\omega,D)\cong\g(V,\kappa,D)$. In particular $D$ is injective.
The following proposition characterizes non-trivial semi-equicontinuous 
orbits, according to our new perspective, 
in terms of the data $(V,\kappa, D)$. 

\begin{prop} \mlabel{prop:char2} 
The coadjoint orbit of $\lambda = (z^*, \alpha, t^*)$ in 
$\g(V,\kappa,D)$ is non-trivial and semi-equicontinuous if and only 
if $z^* \not=0$ and $\alpha\res_{D(V)}$ is $\kappa$-bounded. 
If $D(V)$ is dense in $V$, then the 
condition that $\alpha\res_{D(V)}$ is $\kappa$-bounded is 
equivalent to the $\kappa$-boundedness of $\alpha$. 
\end{prop}

\begin{prf} We have to reformulate the conditions 
(SE1-3) from Proposition~\ref{prop:char} in the new context. 
Condition (SE1) that $z^*\not=0$ remains the same. 

For (SE2), we observe that for $\alpha \in V'$ we have 
$(D^*\alpha)(v) = \alpha(Dv)$, and (SE2) is equivalent to the existence 
of some $C > 0$ with 
\begin{equation}
  \label{eq:esti}
|\alpha(Dx)| \leq C \|x\|_Q = C \|Dx\|_{\kappa} 
\end{equation}
(cf.\ \eqref{eq:kappa0}). 
This in turn is equivalent to the $\kappa$-boundedness of 
$\alpha\res_{D(V)}$, which in turn is 
equivalent to the existence of some 
$\tilde\alpha$ in the $\kappa$-dual of $V$ such that 
$D^*(\tilde\alpha - \alpha) = 0$.  This condition implies that 
$\tilde\alpha - \alpha$ is fixed by $\gamma(\R)$, so that 
we have for each $s \in \R$ the relation 
$\gamma(s)^*\tilde\alpha - \tilde\alpha = \gamma(s)^*\alpha - \alpha,$
and thus 
$$ \gamma(\R)^*\tilde\alpha - \tilde\alpha = \gamma(\R)^*\alpha - \alpha.$$
It follows in particular that $\gamma(\R)^*\alpha$ is 
semi-equicontinuous if and only if $\gamma(\R)^*\tilde\alpha$ has 
this property. But $\gamma(\R)$ consists of $\kappa$-isometries, so that  
$\gamma(\R)^*\tilde\alpha$ is equicontinuous in the $\kappa$-dual 
of $V$, hence also equicontinuous in $V'$. Therefore 
(SE3) is a consequence of (SE2). 
\end{prf}

\begin{rem} (a) In the preceding proof, the difference 
$\tilde\alpha - \alpha \in [\g,\g]^\bot \subeq \g'$ 
is a fixed point of the coadjoint 
action. Therefore 
$$ \cO_{\lambda + \tilde \alpha - \alpha}
= \tilde \alpha - \alpha + \cO_{\lambda}. $$
From the preceding argument we obtain the inclusion 
$$ \g'_{\rm seq} \subeq (\R \times V_{\kappa}' \times \R) 
+ [\g,\g]^\bot, $$
i.e., by a shift with a fixed point, every semi-equicontinuous coadjoint 
orbit can be moved into the subspace 
$\R \times V_{\kappa}' \times \R \subeq \g'$. In particular, they 
are contained in the dual space of any topological Lie algebra 
$\tilde\g$ between $\g$ and its $\kappa$-completion 
$\R \oplus  V_\kappa \oplus  \R$. 

(b) The proof of the preceding proposition implies in particular 
that condition (SE3) in Proposition~\ref{prop:char} is redundant. 
Assuming (C1)-(C3), it follows from (SE1) and (SE2). 

(c) If $D(V)$ is dense, then \eqref{eq:esti} implies 
that $\|\alpha\|_\kappa \leq C$. 
In particular, $\alpha\res_{D(V)}$ is $\kappa$-bounded if and only 
if $\|\alpha\|_\kappa < \infty$. 

In the situation of Remark~\ref{rem:4.3} we have seen that $D(V)$ is dense in 
$V_{\kappa}$, hence in $V$ with respect to $\kappa$. 
Therefore the $\kappa$-boundedness of 
$\alpha\res_{D(V)}$ implies that it has 
exactly one $\kappa$-continuous extension to $V$. 
Whether this extension coincides with $\alpha$ or not depends on 
the density of $D(V)$ in $V$. If $D(V)$ is not dense in $V$, 
then $\alpha\res_{D(V)}$ has many different continuous extensions 
to~$V$, some of which are not $\kappa$-bounded 
(cf.\ Example~\ref{ex:no-eigenvalue} below). Moreover in this case it may also happen
that there exists a different $\kappa_1\neq \kappa$ on $V$ such that the Lie algebras $\g(V,\kappa,D)$ and $\g(V,\kappa_1,D)$ are the same (cf.\ Example~\ref{ex:6.3}(d) below).
\end{rem}

\begin{ex} \mlabel{ex:no-eigenvalue} 
(a) For $\cH = L^2([0,1],\C)$, $V = C([0,1],\C)$ and 
\[ (\gamma(t)f)(x) = e^{itx} f(x) \] 
we obtain $(Df)(x) = ixf(x)$. Clearly, $\ker D =0$ and 
$\oline{\im(D)} = (\ker D)^\bot =\cH$. 
However, $D(V) \subeq \{ f \in V \: f(0) = 0\}$ is not dense in 
$V$ and $\ev_0 \: V \to \C$ is a $\gamma$-invariant continuous 
linear functional. 

In this case 
\[ \omega(f,g) = \la Df, g \ra = i \int_0^1 x f(x) \oline{g(x)}\, dx\]  
and 
\[ \omega(Df,g) = \la D^2f, g \ra = - \int_0^1 x^2 f(x) \oline{g(x)}\, dx\]  
is negative definite. 

(b) On $\R$ we consider the measure 
$\mu = \delta_0 + \chi_{[0,\infty[} \cdot dx,$
so that 
$$ L^2(\R,\mu) = \C \oplus L^2(\R_+,dx). $$ 
On this space we consider the skew-adjoint operator 
$$ \oline D f(x) := i x f(x) $$
and the domain 
$$ V := C_c^\infty([0,\infty[) 
= \{ (f(0), f) \in \C \times L^2(\R_+) \: f \in C_c^\infty([0,\infty[)\}. $$
Then $V$ is dense in $L^2(\R,\mu)$ and $V$ is invariant 
under the corresponding unitary one-parameter group 
$$(\gamma(t)f)(x) = e^{itx}f(x) $$
and $\oline D$ itself. Therefore $V$ is a core of $\oline D$, 
i.e., $\oline D$ is the closure of 
$D := \oline D\res_{V}$. Here $0$ is an eigenvalue of $\oline D$, 
but not of $D$. 
\end{ex}

\begin{rem}
 Since $\gamma(\R) \subeq \GL(V)$ preserves the scalar product 
$\kappa$, it defines an orthogonal one-parameter group 
$\gamma \: \R \to \OO(V_{\kappa})$ 
whose infinitesimal generator $\gamma'(0)$ extends $D$. Since 
$V \subeq V_{\kappa}$ is a $\gamma(\R)$-invariant dense subspace, 
\cite[Prop.~II.1.7]{EN00} implies that $\gamma'(0) = \oline D$ 
is the closure of $D$ in $V_{\kappa}$. 
As $V$ is invariant under $D$, we clearly 
have $V \subeq \cD^\infty(\oline D) = \bigcap_{n = 1}^\infty \cD(\oline D^n)$, 
the space of smooth vectors for $\oline D$ resp.,~$\gamma$. 
\end{rem}

\begin{prop} \mlabel{prop:real-stone} {\rm(Real version of Stone's Theorem)} 
Let $\cH$ be a real Hilbert space, 
$\gamma \: \R \to \OO(\cH)$ be a strongly continuous 
one-parameter group of isometries of $\cH$, and  
$\cD \subeq \cH$ be a dense subspace invariant under 
$\gamma(\R)$ and consisting of $C^1$-vectors. 
Then the operator 
$$ A\: \cD \to \cH, \quad A v := \gamma'(0)v \quad  \mbox{ satisfies } \quad 
\oline A = \gamma'(0) = - A^\top. $$
\end{prop}

\begin{prf} \cite[Prop.~II.1.7]{EN00} implies that 
$\gamma'(0) = \oline A$. Further, 
$A \subeq -A^\top$ follows from the fact that $\gamma(\R)$ 
consists of isometries. From the closedness of 
$A^\top$ (\cite[Thm.~13.9]{Ru73}), 
we further derive that $\oline A \subeq -A^\top = - \oline A^\top$. 

Applying Stone's Theorem (\cite[Thm.~13.37]{Ru73}) 
to the complex linear extension 
$\gamma_\C(t)$ of the operators $\gamma(t)$ to the complex 
Hilbert space $\cH_\C$, we derive that 
$\gamma_\C'(0)$ is skew-adjoint. Clearly, $\gamma_\C'(0)$ 
is the complex linear extension of the operator 
$\gamma'(0)$ on $\cH$, which implies that 
$\gamma'(0) = - \gamma'(0)^\top$. 
In particular, the domain of $A^\top$ is not larger than 
$\cD(\oline A) = \cD(\gamma'(0))$, and this leads to 
$A^\top = \oline A^\top = - \oline A$. 
\end{prf}

\begin{rem} \mlabel{rem:4.9} 
We take a closer look at the continuous injection 
$$ \eta \: V \to V_Q, \quad Q(\eta(x),y) = \omega(x,y). $$
Here we identify $V$ with a dense subspace of the real Hilbert 
space $V_Q$ and consider $D$, as a well as $\eta$, as unbounded 
operators with domain $\cD = V$. 

We have already seen in the proof of Theorem~\ref{thm:4.2} that 
\begin{equation}
  \label{eq:eta-d-rel}
\eta \circ D = \id_V, 
\end{equation}
which implies in particular that $V \subeq \eta(V)$. 
As the subspace $V \subeq V_Q$ is invariant 
under the strongly continuous 
one-parameter group $\gamma(\R)$ of isometries of $V_Q$, 
its infinitesimal generator in $V_Q$ satisfies 
$$ \gamma'(0) = \oline D = - D^\top $$
(Proposition~\ref{prop:real-stone}). 
The operator $D \: V \to V_Q$ is injective, so that we may 
define its inverse by 
$$ \cD(D^{-1}) \: D(V) \to V, \quad Dv \mapsto v. $$
The relation \eqref{eq:eta-d-rel} now implies that 
$\eta$ is an extension of $D^{-1}$, which is defined on all 
of $V$.

As $\gamma$ defines a smooth $\R$-action on $V$ and 
$\eta$ is continuous, $\eta(V)$ consists of smooth vectors 
for $\gamma$. For $v,w \in V$, the relation 
$$ Q(\eta v, Dw) 
= \omega(v,Dw) 
= - \omega(Dv,w) 
= -  Q(\eta D v,w) =  -Q(v,w) $$ 
implies that $\oline D \eta(v) = - D^\top \eta(v) = v$, i.e., 
\begin{equation}
  \label{eq:olined-q-rel}
\oline D \circ \eta = \id_V. 
\end{equation}
This shows in particular that 
$V \subeq \im(\oline D)$, and since $D(V)$ is dense in 
$\im(\oline D)$, it follows that 
$D(V)$ is dense in $V_Q$. From that we further derive 
$$\ker(\oline D) = \ker(D^\top) = \im(D)^\bot = \{0\}, $$
i.e., that $\oline D$ is injective. 
For $v \in V$ we derive from 
$\oline D \oline D^{-1} v = v = \oline D \eta(v)$
that 
\begin{equation}
  \label{eq:dinv}
\oline D^{-1}\res_V = \eta. 
\end{equation}

On the dense subspace $V^1 := \im(\oline D)$ of $V_Q$ we 
now define 
$$ \omega_1(v,w) := Q(\oline D^{-1}v,w), $$
and note that, by \eqref{eq:4.1} and \eqref{eq:dinv}, $\omega_1$ extends $\omega$. 
Further 
$$ \omega_1(\oline D w, \oline D v) 
= Q(w, \oline D v) = - Q(v, \oline D w) 
= -\omega_1(\oline D v, \oline D w)  $$
implies that $\omega_1$ is skew-symmetric. 
It is continuous with respect to the norm defined by
$$ \|v\|_*^2 = \|v\|_Q^2 + \|\oline D^{-1}v\|_Q^2. $$
Note that 
$$ \|\oline D^{-1}v \|_Q^2 = \|\eta(v)\|_Q^2 = \kappa(v,v) 
\quad \mbox{ for } \quad v \in V. $$

As $\eta \: V \to V_Q$ is continuous, we further observe that 
$\eta(V) \subeq \cD^\infty(\oline D)$ and the map 
$V \to \cD^\infty(\oline D)$
is continuous with respect to the natural Fr\'echet topology on 
the right hand side.  
\end{rem}

In the preceding remark we have used the 
scalar product $Q$ on $V$ to obtain a Hilbert space environment 
for $V$. Alternatively, one may consider the scalar product 
$\kappa(v,w) = Q(\eta(v),\eta(w)).$

\begin{lem} \mlabel{lem:4.9} The map 
$$ \eta \: (V,\kappa) \to (\eta(V), Q) $$
is isometric and  extends to an isomorphism 
$\hat\eta \: V_{\kappa} \to V_Q$
of real Hilbert spaces. It intertwines the unbounded 
operator $D \: V \to V_{\kappa}$
with the operator 
\[  \oline D\res_{\eta(V)} = \eta^{-1} \: \eta(V) \to V \subeq V_Q. \]
The closure of $D$ on $V_\kappa$ is injective with dense range. 
\end{lem}

\begin{prf} The relation 
$\eta \circ D = \id_V$ leads to 
\[ \eta \circ D \circ \eta^{-1} = \eta^{-1}\res_{\eta(V)} = 
\oline D \res_{\eta(V)}. \] 

Since $V$ is a core for $D$ in $V_{\kappa}$ and 
$\eta(V) \supeq V$ is a core for $D$ in $V_Q$, the map 
$\hat\eta$ intertwines the closure $\oline D^{\kappa_0}$ 
with the closure $\oline D = \oline D^{Q}$. 
From the corresponding properties of $\oline D$ 
(Remark~\ref{rem:4.9}), it now follows that 
$\oline D^{\kappa_0}$ is injective and has dense range. 
\end{prf}

\section{Complex structures} 
\mlabel{sec:5}

It is well known that symplectic complex structures satisfying a positivity 
condition are the natural key to unitary representations of the 
Heisenberg group $\Heis(V,\omega)$ (cf.\ \cite[p.~455]{FS00}). 
Therefore we study in this section how suitable complex structures 
on $V$, resp., a slightly enlarged space, can be obtained from the available 
structure defined by the skew-symmetric operator $D$ on $(V,\kappa)$.

\begin{lem} \mlabel{lem:6.1}
Let $\cH$ be a complex Hilbert space and $\gamma:\R\rightarrow \UU(\cH)$ be a strongly continuous one-parameter group with selfadjoint generator $A = -i \gamma'(0)$. 
Equip $\cD^\infty(A)$ with the $C^\infty$-topology 
for which the inclusion $\cD^\infty(A) \to \cH^{\N_0}, 
v \mapsto (A^nv)_{n \in \N_0}$ is a topological embedding. 
Then the following assertions hold:
\begin{description}
\item[\rm(a)] $\cD^\infty(A)$ is a Fr\'echet space and the action $\R\times\cD^\infty(A)\rightarrow\cD^\infty(A),(t,v)\mapsto\gamma(t)v$ is smooth.
\item[\rm(b)] Let $V\subset \cD^\infty(A)$ be a real subspace invariant under $\gamma$ and dense in $\cD^\infty(A)$ w.r.t. the Hilbert norm. Then $V$ is also dense in $\cD^\infty(A)$ w.r.t. the $C^\infty$-topology.
\end{description}
\end{lem}

\begin{proof}
For (a) see \cite[Rem.~4.7(a) and Cor.~4.5]{Ne10b}.

(b) Let $v \in \cD^\infty(A)$ and choose a sequence $(v_n)$ in $V$ converging to 
$v$ in $\cH$. For $f \in C_c^\infty(\R,\R)$ the bounded operators 
\[ \gamma(f)w = \int_\R f(t)\gamma(t)w\, dt, \quad w \in \cD^\infty(A), \] 
satisfy 
\begin{align*}
A\gamma(f)w 
&= \int_\R f(t)A\gamma(t)w\, dt 
= -i \int_\R f(t)\gamma'(t)w\, dt \\
&= i \int_\R f'(t)\gamma(t)w\, dt = i \gamma(f')w, 
\end{align*}
and by iteration we obtain $A^k \gamma(f) = i^k \gamma(f^{(k)})$. 
We thus obtain 
\[ A^k \gamma(f)v_n 
= i^k \gamma(f^{(k)})v_n  \to i^k \gamma(f^{(k)})v 
= A^k \gamma(f)v.\] 
This shows that $\gamma(f)v$ is contained in the closure 
$\overline{V}$ of $V$ in $\cD^\infty(A)$ in the $C^\infty$-topology. Here we have used that the Fr\'echet space $\overline{V}$ is invariant under $\gamma(f)$. In view of (a), every neighborhood of $v$ in $\cD^\infty(A)$ 
contains an element of the form $\gamma(f)v$, so that $v \in \oline V$. 
\end{proof}

\begin{ex} \mlabel{ex:6.3} (a) Let $\cH$ be a complex Hilbert space and 
$A = A^*$ be a self-adjoint operator with the domain 
$\cD \subeq \cH$, for which the dense subspace 
$V$ is invariant under the action of the 
corresponding unitary one-parameter group 
$\gamma(t) := e^{itA}$ and also under $A$. 
Then $V$ is a core for $A$ consisting of smooth vectors and 
dense in $\cD^\infty(A)$ (cf.\ \cite[Prop.~II.1.7]{EN00}). 

On $V$ we have the real scalar product 
$\kappa(v,w) := \Re \la v,w \ra$, and 
$D := iA\res_{V}$ is a skew-symmetric operator, defining 
the alternating form 
$$ \omega_D(x,y) := \kappa(Dx,y) 
= \Re \la iAx,y \ra 
= -\Im \la Ax,y \ra. $$
We further have 
$$ \omega_D(Dx,y) = -\Im \la A(iA)x,y\ra  
= -\kappa(A^2x,y) =  -\kappa(Ax, Ay), $$
so that $Q(x,y) := -\omega_D(Dx,y)$ is positive definite 
if and only if $\ker(A)\cap V = \{0\}$. 

(b) A natural topology on $V$ is the $C^\infty$-topology.
This topology turns $\g_A = \heis(V,\omega_D) \rtimes_D \R$ into a 
locally convex Lie algebra that integrates to a Lie 
group $G_A = \Heis(V,\omega_D) \rtimes_\gamma \R$ which in the case $V=\cD^\infty(A)$ is a Fr\'echet--Lie group (cf.\ Lemma~\ref{lem:6.1}(a)).

(c) A typical class of examples arises for 
$V= \cD^\infty(A)$, the space of all smooth vectors of $A$. 
In this case the subspace $AV$ of $V$ is dense with respect to the 
Fr\'echet topology if and only if the space $\cD^{-\infty}(A) 
:= \cD^\infty(A)'$ of distribution vectors contains no $\gamma$-invariant 
element. This is not always the 
case, even if $\ker A$ is trivial. 

(d) Let us recall Example~\ref{ex:no-eigenvalue}(b). Here we clearly are in the situation of Example~\ref{ex:6.3}(a) and hence obtain a double extension $\g(V,\kappa_1,D)$, where $\kappa_1(v,w):=\Re\langle v,w\rangle$ and $V$ is equipped with the $C^\infty$-topology. Furthermore $D$ is injective on $V$ but $D(V)$ is not dense in $V_{\kappa_1}$. On the other hand, by Remark~\ref{rem:4.3}, there exists a continuous real scalar product $\kappa$ on $V$ such that $\kappa(Dx,y)=\kappa_1(Dx,y)$ holds for $x,y \in V$ and $D(V)$ is dense in $V_\kappa$. In particular, the real scalar products $\kappa$ and $\kappa_1$ are different but the oscillator algebras $\g(V,\kappa_1,D)$ and $\g(V,\kappa,D)$ coincide.
\end{ex}

\begin{lem} \mlabel{lem:6.2} 
Suppose that $A \: \cD \to \cH$ is an essentially selfadjoint 
operator on the complex Hilbert space $\cH$. Then 
$\ker(\oline A) = \{0\}$ is equivalent to the density of $\im(A) = 
A(\cD)$ in $\cH$. 
\end{lem}

\begin{prf} This follows from the relation 
$\ker(\oline A) = \im(\oline A)^\bot = \im(A)^\bot$ 
(\cite{Ru73}), 
which in turn follows from the density of $\im(A)$ in $\im(\oline A)$. 
\end{prf}

In general $\ker A = \{0\}$ does not imply that 
$\ker \oline A = \{0\}$ 
(cf.\ Example~\ref{ex:no-eigenvalue}(b)).

\begin{rem} \mlabel{rem:complex}
Now let $(V, \kappa)$ be a real pre-Hilbert space and 
$D:V\rightarrow V$ be a skew-symmetric operator which integrates to a continuous 
orthogonal one-parameter group $\gamma$ on~$V$. 
We assume that $D(V)$ is dense in $\cH :=V_\kappa$ 
(cf.\ Remark~\ref{rem:4.3}). 
Then $V_\C \subeq \cH_\C$ is invariant under the unitary 
one-parameter group $\gamma$ on $\cH_\C$, hence a core for~$D$. 
Now the operator $A := -iD$, defined on $V_\C$, is essentially 
self-adjoint and the canonical antilinear 
involution $\sigma$ of $\cH_\C$ 
with $(\cH_\C)^\sigma = \cH$ satisfies 
$$ \sigma \circ \oline A = - \oline A \circ \sigma, $$
so that $\sigma P(E) \sigma = P(-E)$ holds for the spectral 
measure $P$ of $\oline A$. 
The assumption that $D(V)$ is dense in $\cH$ implies that 
$P(\{0\})= 0$ (Lemma~\ref{lem:6.2}), so that $\cH_\C$ decomposes into two closed subspaces 
$$ \cH_\C = \cH_+ \oplus \cH_-, \quad \mbox{ where } \quad 
\cH_\pm = P(\pm ]0,\infty[)\cH_\C. $$
In view of $\sigma(\cH_\pm) = \cH_\mp$, we obtain an isometric real linear 
isomorphism 
$$ p \: \cH_+ \to \cH, \quad v \mapsto \frac{1}{\sqrt 2}(v + \sigma(v)). $$
This leads to the structure of a complex Hilbert space on $\cH$, where 
the complex structure is given by 
$$ I p(v) := p(iv) \quad \mbox{ for } \quad  v \in \cH_+. $$
The scalar product on $\cH$ satisfies 
$$\Re\langle x,y \rangle = \kappa(x,y) \quad \mbox{ for } \quad x,y\in \cH$$ 
and $\gamma$ becomes a unitary one-parameter group on $\cH$.
We define the positive symmetric operator $A_\R$ on $\cH$ by 
$A_\R p(v) := p(Av)$ for $v \in \cH_+ \cap \cD(A)$ 
and note that 
$$ \tilde \cD := \{ v \: p^{-1}(v) \in \cD(A)\} 
= p(\cH_+ \cap \cD(A)). $$
For $p(v) \in V$, we then have 
$$ Dp(v) = p(Dv) = p(iAv) = I p(Av) = I A_\R p(v), $$
so that 
$$ D = I A_\R \quad \mbox{ and } \quad A_\R = -ID \geq 0 $$ 
hold on $V$. 
\end{rem}

\begin{rem} \mlabel{rem:5.4} 
In general the subspace $V_\C$ does not need to be adapted 
to the decomposition $\cH_\C = \cH_+ \oplus \cH_-$, but if 
$V = \cD^\infty(D)$, 
then $V_\C = \cD^\infty(A)$ is invariant under the 
spectral projections $P(\pm ]0,\infty[)$. 
\end{rem}

\begin{defn}
Let $(V,\omega)$ be a symplectic vector space. A complex structure $I$ on $V$ is called {\it $\omega$-positive} if 
$\omega(Iv,w)$ is a positive definite symmetric bilinear form and 
$I$ is called  {\it $\omega$-negative} if $-I$ is $\omega$-positive.
\end{defn}

We summarize the preceding discussion in the following proposition.

\begin{prop} \mlabel{prop:propcomplstr} 
Let $(V, \kappa)$ be a real pre-Hilbert space and 
$D:V\rightarrow V$ be a skew-symmetric operator which integrates to a continuous 
orthogonal one-parameter group $\gamma$ on $V$. Assume that $D(V)$ is dense in $V_\kappa$. 
Then there exists a $\kappa$-orthogonal 
$\omega_D$-negative complex structure $I$ on $\cD^\infty(\overline{D}) \subset V_\kappa$ commuting with $\gamma(\R)$.
\end{prop}

\begin{proof}
We recall the complex Hilbert structure on $V_\kappa$ 
from Remark~\ref{rem:complex}. Obviously $\cD^\infty(\overline{D})=\cD^\infty(A_\R)$ is a complex subspace. The corresponding complex structure $I$ on $\cD^\infty(\overline{D})$ is $\omega_D$-negative because $A_\R$ is positive and
$$\omega_D(Iv,w)=-\kappa(A_\R v,w)  \quad \mbox{ for } \quad  v,w \in \cD^\infty(\overline{D}).$$
Since $\gamma$ acts on $V_\kappa$ by unitaries, $I$ commutes with $\gamma(\R)$. That $I$ is $\kappa$-isometric follows 
from the construction in Remark \ref{rem:complex}, but it also follows from the fact that $I$ is $\omega_D$-negative and $D(V)\subset V_\kappa$ is dense.
\end{proof}

\begin{rem} \mlabel{rmk:complexunique}
Let $G(V,\omega,\gamma)$ be a generalized oscillator group for which 
$Q(x,y)=\omega(Dx,y)$ is positive definite. Then there exists at most one $\omega$-positive complex structure $I$ on $V$ such that $I$ commutes with $\gamma(\R)$ and $D(V)$ is a dense subset of $V$ in the norm $\|x\|_I=\sqrt{\omega(Ix,x)}$. For a proof see \cite[Scholium~3.3]{BSZ92}.
In particular, in Proposition \ref{prop:propcomplstr} the $\omega_D$-negative complex structure commuting with $\gamma(\R)$ is unique.
\end{rem}

\begin{ex} A simple instructive example is the following. 
We consider $V = \R^2$ with the canonical scalar product $\kappa$ and 
$$ D = \pmat{0 & -d \\ d & 0}, \quad d > 0, \quad 
A = -i D = \pmat{0 & id \\ -id & 0}.$$
Then $\omega_D(x,y) = \la Dx,y \ra = d(x_1 y_2 - x_2 y_1)$ and 
$$ A(e_1 \mp i e_2) = \pm d(e_1 \mp i e_2), $$
so that 
$\cH_\pm = \C(e_1 \mp i e_2).$ Hence the $\omega_D$-negative complex structure on $V$ which commutes with $D$ is given by 
$Ie_1=e_2$.
\end{ex}

\begin{ex} Consider the case where $\g=\g(V,\omega,D$) is a direct 
limit oscillator algebra, i.e., let $V_n\subset V$ be subspaces with
$$V=\bigcup_n V_n, \quad V_n \subset V_{n+1} \quad \mbox{ and } 
\quad \dim V_n <\infty$$
such that $\omega\vert_{V_n\times V_n}$ is non-degenerate and $D(V_n)\subset V_n$ holds for all $n\in \N$. We equip $\g$ with the direct limit topology. For $n\in \N$ we set $D_n:= D\vert_{V_n}$.

First we consider the conditions (C1-3) of Theorem~\ref{thm:char}. Suppose (C1) holds, i.e., $Q(x,y)=\omega(Dx,y)$ is positive definite. Then $D$ is injective and $D_n^{-1}$ exists for all $n\in\N$, hence also $D^{-1}:V\rightarrow V$ exists. Now $\omega(x,y)=Q(D^{-1}x,y)$ implies that $i_x\omega$ is $Q$-continuous, i.e., (C2) holds. Moreover $Q(\eta(x),y)= Q(D^{-1}x,y)$ for $x,y \in V$ implies $\eta=D^{-1}$, in particular $\eta$ takes values in $V$. Since the linear operator $D^{-1}$ is continuous with respect to the direct limit topology on $V$, condition (C3) also holds.  
Therefore $\g$ has a non-trivial semi-equicontinuous coadjoint orbit 
if and only if, after possibly replacing $D$ by $-D$, $Q$ is positive definite. 

Suppose that $Q$ is positive definite. The operator $-D^2$ is positive for the real scalar product $\kappa(x,y)=\omega(x,D^{-1}y)$ on $V$ and hence we obtain a complex structure $I_n:=\frac{D_n}{\sqrt{-D_n^2}}$  on each finite dimensional $V_n$. By construction the $I_n$
 fit together to a complex structure $I$ on $V$ satisfying $I\vert_{V_n}=I_n$. 
With the scalar product $\langle x,y \rangle := \kappa(x,y)-i\kappa(Ix,y)$ we turn $V$ into a complex pre-Hilbert space. Applying the Spectral Theorem 
to the positive operators $A_n:=-I_nD_n$ on the complex subspaces $V_n$ yields an orthonormal Hamel basis $\tilde e_i, i\in\N$ of $V$ with $D \tilde e_i=a_iI \tilde e_i, a_i>0$. By setting $e_i:=\frac{1}{\sqrt{a_i}} \tilde e_i$, $f_i:=-Ie_i$ we obtain a symplectic basis $(e_i,f_i)$ of $V$ with
\begin{align} \label{eq:sympldiag}
De_i=-a_i f_i, \quad Df_i=a_i e_i, \quad a_i>0.
\end{align}
The Hamiltonian takes the form $H_D(v)=\frac{1}{2}\omega(Dv,v)= \frac{1}{2}\sum_i a_i(p_i^2+q_i^2)$ for $v=\sum_i p_i e_i+q_i f_i$. Conversely any symplectic vector space with symplectic basis $(e_i,f_i), i\in\N$, equipped with the direct limit topology and $D$ defined by \eqref{eq:sympldiag} defines an oscillator algebra with semi-equicontinuous orbits.

\end{ex}

\section{Unitary representations} 
\mlabel{sec:6}

The purpose of this section is to show that the Fock--Schr\"odinger   
representation is semibounded for the oscillator group $G = G(V,\omega,\gamma)$ 
with semi-equicontinuous orbits. 
This brings us full circle because it implies 
that the existence of 
non-trivial semibounded unitary representations is equivalent to the 
existence of non-trivial semi-equicontinuous coadjoint orbits.

\subsection{The Fock--Schr\"odinger representation} 

Let $V$ be a complex vector space and 
$\beta \: V \times V \to \C$ be a positive semidefinite hermitian form on
$V$. We define the associated {\it Heisenberg group} 
as 
$\Heis(V,\beta) := \R \times V$ with the product 
$$ (t,v) (s,w) =  \big(t + s - \Im \beta(v,w),v + w\big). $$
For the alternating form $\omega := - 2 \Im \beta$, this means that 
$$ (t,v) (s,w) =  \big(t + s + \frac{1}{2}\omega(v,w),v + w\big). $$
The hermitian form $\beta$ can be reconstructed from $\omega$ and 
the complex structure $I$ by 
\begin{align*}
\beta(v,w) 
&= \Re \beta(v,w) + i \Im \beta(v,w)
= \Im\beta(Iv,w) - \frac{i}{2}\omega(v,w) \\
&= -\frac{1}{2} \omega(Iv,w) - \frac{i}{2}\omega(v,w) 
= -\frac{1}{2}\big(\omega(Iv,w) + i\omega(v,w)\big). 
\end{align*}
This formula shows in particular that, conversely, 
any complex structure $I\in \Sp(V,\omega)$ for which 
$\omega(Iv,v)$ is negative semidefinite leads to a 
positive semidefinite hermitian form $\beta$ on $V$.

The Heisenberg group $\Heis(V,\beta)$ 
acts on $V$ by $(v,t). w := v + w$. 

\begin{prop} \mlabel{prop:fockrep} {\rm(\cite[Prop.~II.4.6]{Ne00})} 
The operators on $\C^V$ defined by 
$$ \big(\pi(t,v)f\big)(z) = e^{it-\frac{1}{2}\beta(v,v)+ \beta(z,v)} f(z-v)
\quad \mbox{ for } \quad g = (t,v) \in \Heis(V,\beta)$$ 
define a representation of $\Heis(V,\beta)$ on $\C^V$ satisfying the Weyl
relations 
$$ \pi(0,v)\pi(0,w) = e^{-i \Im \beta(v,w)} \pi(0,v + w)
= \pi\big(-\Im \beta(v,w),v + w \big). $$
Further, $K(z,w) := e^{\beta(z,w)}$ is a positive definite 
kernel on $V$, and the action of $\Heis(V,\beta)$ 
leaves the corresponding reproducing kernel 
Hilbert space ${\cal H}_K$ invariant 
and defines an irreducible unitary representation on this space. 
\end{prop}

\begin{rem}
The constant function $1 = K_0 \in \cH_K$ is a unit vector and 
$$\phi(t,v) := 
\la \pi(t,v)1,1\ra = (\pi(t,v)1)(0) 
= e^{it-\frac{1}{2}\beta(v,v)} 
= e^{it+\frac{1}{4}\omega(Iv,v)}. $$
This formula expresses quite naturally 
how each $\omega$-negative complex structure~$I$  
leads to a positive definite function on $\Heis(V,\omega)$. 
\end{rem}

\begin{rem} \mlabel{rem:densekernel}
Suppose that $\beta$ is positive definite and let $V_\beta$ be the Hilbert space obtained by completion of $(V,\beta)$. We consider the kernel $\hat K(v,w):=e^{\beta(v,w)}$ on $V_\beta$. Then $K=\hat K\vert_{V\times V}$. Since the kernel $K$ on $V$ is jointly continuous with respect to the $\beta$-norm and $V$ is dense in $V_\beta$, we have an isomorphism of Hilbert spaces $$\phi:\cH_{\hat K} \cong \cH_K, f \mapsto f\vert_V.$$ In particular we obtain for the corresponding Fock-Schr\"odinger representations 
$$\pi:\Heis(V,\beta)\rightarrow \UU(\cH_K), \qquad \hat\pi:\Heis(V_\beta,\beta)\rightarrow \UU(\cH_{\hat K})$$ that $\pi\cong \hat \pi\vert_{\Heis(V,\beta)}$ via $\phi$.
\end{rem}

\begin{prop}\mlabel{prop:generalfockrep}
Let $\gamma$ be a strongly continuous  unitary one-parameter group on the complex pre-Hilbert space $(V,\beta)$. Then the Fock-Schr\"odinger representation $\pi:\Heis(V,\beta)\rightarrow\UU(\cH_K)$ extends to a continuous unitary representation of the topological group $\Heis(V,\beta)\rtimes_\gamma\R$ defined by 
$$\pi_\gamma:\Heis(V,\beta)\rtimes_\gamma\R\rightarrow\UU(\cH_K),\quad \pi_\gamma(t,v,s)f=\pi(t,v)(f\circ\gamma(s)^{-1}).$$ 
If $A:=i\gamma'(0)\geq0$ then $-i\dd\pi_\gamma(0,0,1)\geq0$, where $\dd\pi_\gamma(0,0,1)$ denotes the skew-adjoint generator of $s \mapsto \pi_\gamma(0,0,s)$.
\end{prop}

\begin{proof} Although this is more or less well-known, for the sake of 
completeness, we provide a proof. 
As the kernel $K$ is invariant under $\gamma$, the operators $\pi_\gamma(0,0,s)$ are unitary and it is readily verified that $\pi_\gamma$ defines a representation. Now suppose that $A\geq0$. Let $V_\beta$ denote the completion of $(V,\beta)$. Note that $i\gamma'(0)\geq0$ on $V$ implies $i\gamma'(0)\geq0$ on $V_\beta$ since $V$ is $\gamma(\R)$-invariant. 

In order to prove the second assertion, in view of Remark~\ref{rem:densekernel}, 
we may assume that $V=V_\beta$. 
We consider the complex semigroup $\C^+=\R+i\R_{\geq0}$ with involution $z^*:=-\bar{z}$ for $z\in\C^+$. Then $\hat{\gamma}(z):=e^{iAz^*}, z \in \C^+$ defines an extension of $\gamma$ to a continuous involutive semigroup representation $\hat{\gamma}:\C^+ \rightarrow B(V)$ by contractions which is anti-holomorphic on $\text{int}(\C^+)$. For $f\in\cH_K$ and $z\in\C^+$ we define $\alpha(z)f=f\circ\hat{\gamma}(z^*)$. Then $\alpha(t)=\pi_\gamma(0,0,t)$ for $t\in \R$. Since $\alpha(z)K_w=K_{\hat{\gamma}(z)w}$ holds, the linear operators $\alpha(z)$ leave the dense subset $\cH_K^0:=\Spann \{ K_v : v\in V\}$ invariant and define an involutive semigroup representation $\alpha:\C^+\rightarrow \End(\cH_K^0)$. By \cite[Thm.~II.4.4]{Ne00} we obtain
$$\|(\alpha(z)\vert_{\cH_K^0}\|^2= \sup_{v\in V}e^{\beta(\hat{\gamma}(z^*)v,\hat{\gamma}(z^*)v)-\beta(v,v)}\leq1.$$
With \cite[Prop.~II.4.9]{Ne00} we conclude that $\alpha(z)$ is a contraction with  $\alpha(z)(\cH_K)\subset\cH_K$. This leads to a continuous involutive semigroup representation $\alpha:\C^+\rightarrow B(\cH_K)$ by contractions. Since $z\mapsto \langle \alpha(z)K_v,K_w \rangle=e^{\beta( w,\hat{\gamma}(z)v)}$ is holomorphic on $\text{int}(\C^+)$ for all $v,w\in V$, 
the representation $\alpha$ is also holomorphic on $\text{int}(\C^+)$ by 
\cite[Lem.~IV.2.2]{Ne00}. 
We conclude with \cite[Thm.~9.13(ii)]{HN93} that $-i\dd\pi_\gamma(0,0,1)$ is non-negative.
\end{proof}

\subsection{Application to generalized oscillator groups} 

Let $G := G(V,\omega,\gamma)$ be a generalized oscillator group. 

\begin{defn} We call a generalized oscillator group 
$G(V,\omega,\gamma)$ {\it standard} if there exists a 
unitary representation $(\rho, \hat V)$ of $\R$ such that 
\begin{description}
\item[\rm(a)] $V \cong \hat V^\infty$ is the space of smooth 
vectors for $(\rho, \hat V)$, endowed with its natural Fr\'echet topology. 
\item[\rm(b)] $\gamma(t) = \rho(t)\res_V$ for $t \in \R$. 
\item[\rm(c)] $\omega(v,w) = -2\Re \la \rho'(0)v,w\ra$. 
\item[\rm(d)] $i \rho'(0) \geq 0$. 
\end{description}
\end{defn}

\begin{rem}\mlabel{rem:stdoscrmk} 
(a) Note that for every standard generalized oscillator 
group $G = G(V,\omega,\gamma)$, the operator 
$D = \gamma'(0)$ on $V$ is essentially skew-adjoint 
and satisfies $\rho'(0) = \oline{D}$. 

(b) For the non-negative selfadjoint operator $A := i \rho'(0)$ 
on $\hat V$ we have 
$$ \omega(v,w) = 2\Re \la i A v,w\ra 
= - 2\Im \la Av,w\ra. $$

(c) Since $\omega$ is non-degenerate on $V$ the operator $A\vert_{\hat V^\infty}$ is injective. As $\hat V^\infty = \bigcap_{n = 1}^\infty \cD(A^n)$ contains every vector of $\ker  A \subset \hat V$ we obtain $\ker A=\{0\}$. Thus the image $D(V)$ is a dense subspace of $V$ in the Hilbert norm (c.f. Lemma~\ref{lem:6.2}) and hence by Lemma~\ref{lem:6.1}(b) also with respect to the $C^\infty$-topology.
\end{rem}

\begin{lem}\mlabel{lem:semiboundedcond}
Let $G=G(V,\omega,\gamma)$ be a generalized oscillator group with a non-trivial semi-equicontinuous coadjoint orbit and suppose that 
$D=\gamma'(0)$ has dense range. 
Let $\pi:G(V,\omega,\gamma) \rightarrow \UU(\cH)$ be a smooth representation with $\pi(t,0,0)=e^{it}$ and $-i\dd\pi(0,0,1)$ bounded from below. 
Then $\pi$ is semibounded.
\end{lem}

\begin{proof}
By Remark~\ref{rem:4.3} we may assume that there exists a continuous real scalar product $\kappa$ on $V$ such that $\omega(x,y)=-\kappa(Dx,y)$. Consider the support functional 
$$s_\pi(a)= \sup_{v \in \cH^{\infty}, \|v\|=1} \langle i\dd\pi(a)v, v \rangle \quad \text{ for $a\in\g=\g(V,\omega,D)$.}$$ 
By assumption we have $s_\pi(0,0,1)\leq c < \infty$. Note that 
$$\Ad(0,x',0)(0,0,1)=\Big(\frac{1}{2}\kappa(Dx',Dx'),-Dx',1\Big).$$
Next we note that $s_\pi$ is $\Ad$-invariant and 
a supremum of continuous linear functions. 
From the density of $D(V)$ in $V$, we thus 
obtain $s_\pi\big(\frac{1}{2}\kappa(x,x),x,1\big)\leq c$ for all $x\in V$. We conclude $s_\pi(t,x,1) \leq c-t+\frac{1}{2}\kappa(x,x)$ because of $\pi(t,0,0)=e^{it}$. It follows that $s_\pi$ is bounded on a non-empty open subset of $\g$, i.e., $\pi$ is semibounded.
\end{proof}

\begin{prop}\mlabel{prop:semibfockrep}
Let $G=G(V,\omega,\gamma)$ be a standard generalized oscillator group, $A=i\gamma'(0)$ and $\cH_K \subeq \C^V$ be the Hilbert space with reproducing kernel 
$K(z,w) = e^{\la Az,w\ra}$. Then we obtain by 
\begin{align*}
& (\pi(t,v,s)f)(z) 
= e^{it - \frac{1}{2}\la Av,v\ra + \la Az, v\ra} f(\gamma(s)^{-1}(z-v))
\end{align*}
a semibounded unitary representation of $G$ on $\cH_K$.
\end{prop}

\begin{proof}
Since $G$ is standard, its Lie algebra $\g$ is isomorphic to a double extension and therefore contains a non-trivial semi-equicontinuous coadjoint orbit. The representation $\pi$ is smooth since the 
function $\langle \pi(t,v,s) 1,1 \rangle=e^{it-\frac{1}{2}\langle Av,v \rangle}$ 
is smooth and $1$ is a cyclic vector (\cite[Cor. 7.3]{Ne10b}).  With Proposition~\ref{prop:generalfockrep} and Remark~\ref{rem:stdoscrmk}(c) we obtain that the assumptions of Lemma~\ref{lem:semiboundedcond} are satisfied. Hence $\pi$ is semibounded.
\end{proof}

\begin{thm} For a generalized oscillator group 
$G = G(V,\omega,\gamma)$ with $D:=\gamma'(0)$ 
and $V\not=\{0\}$ the following are equivalent:
\begin{description}
\item[\rm(a)] $G(V,\omega,\gamma)$ has a semibounded unitary 
representation $(\pi, \cH)$ with $\pi(t,0,0) = e^{it}\1$ for $t \in \R$. 
\item[\rm(b)] $G(V,\omega,\gamma)$ has a semibounded unitary 
representation $(\pi, \cH)$ with $[\g,\g] \not\subeq \ker \dd\pi$. 
\item[\rm(c)] $\g(V,\omega,D)$ has a non-trivial semi-equicontinuous orbit.
\item[\rm(d)] There exists a standard generalized oscillator group 
$\hat G := G(\hat V,\hat\omega,\hat\gamma)$ 
and a dense embedding $\iota \: V \to \hat V$, such that 
$$ \iota_G \: G(V,\omega,\gamma) \to G(\hat V,\hat\omega,\hat\gamma), 
\quad (z,v,t) \mapsto  (z,\iota(v),t), $$
possibly after replacing $\gamma$ by $\gamma^{-1}$, is a morphism of Lie groups. 
\end{description}
\end{thm} 

\begin{prf}
(a)$\Rightarrow$(b): follows from $(1,0,0) \in [\g,\g]$. 

(b)$\Rightarrow$(c): Our assumption implies that the momentum set 
$I_\pi$ of $\pi$ is not contained in $[\g,\g]^\bot$. 
Since every coadjoint orbit in $I_\pi$ is semi-equicontinuous, 
(c) follows. 

(c)$\Rightarrow$(d) If $\g(V,\omega,D)$ has a non-trivial semi-equicontinuous 
coadjoint orbit, then Theorem~\ref{thm:4.2} applies, i.e.,  
$\g(V,\omega,D)$ is isomorphic to a double extension $\g(V,\kappa,D)$, 
where $\omega(x,y)=-\kappa(Dx,y)$ (Remark~\ref{rem:4.3}). 
Let $\hat V \subeq V_{\kappa}$ denote the space of smooth 
vectors for the isometric one-parameter group $\hat\gamma$, 
obtained by extending the operators $\gamma(t)$ to isometries 
on $V_{\kappa}$. Note that  $V\subset \hat V$ since $\gamma$ acts smoothly 
on $V$. We endow $\hat V$ with its natural Fr\'echet topology 
defined by the seminorms $p_n(v) := \|D^{n} v\|$, $n \in \N_0$. 
Then 
$$ \hat\omega(v,w) := -\kappa(Dv,w)  $$
is a continuous alternating form on $\hat V$ extending $\omega$, and 
since $DV$ is dense in $V_{\kappa}$ (Lemma~\ref{lem:4.9}), $\hat\omega$ is non-degenerate.  The action of $\hat \gamma$ on $\hat V$ is smooth (c.f. \cite[Thm.~4.4]{Ne10b}) and thus we obtain a generalized oscillator group $G(\hat V, \hat\omega, \hat\gamma)$. Applying Proposition~\ref{prop:propcomplstr} (with $-D$ instead of $D$) yields on $\hat V$ an $\hat \omega$-negative complex structure $I$ commuting with $\gamma(\R)$. Since $I$ is $\kappa$-isometric it may be extended to $V_\kappa$ and by setting
$$\langle v,w\rangle := \frac{1}{2}(\kappa(v,w)-i\kappa(Iv,w)) \quad \mbox{ for } \quad v,w \in 
V_\kappa,$$
we turn $V_\kappa$ into a complex Hilbert space. Then $\hat \gamma$ becomes a unitary one-parameter group on $V_\kappa$.  Furthermore 
$$\hat{\omega}(v,w)= -2\Re\langle Dv,w\rangle  \quad \mbox{ for } 
\quad v,w \in \hat V, $$
and
$$ 0 \geq \hat\omega(Iv,v) = - \kappa(DIv,v) \quad \mbox{ for } \quad v \in \hat V $$
implies that $ID = DI \geq 0$. Thus $G(\hat V, \hat\omega, \hat\gamma)$ is standard. The inclusion $\iota:V\subset \hat V$ is dense by Lemma \ref{lem:6.1}(b) and by construction $\iota_G$ is a Lie group morphism. 

(d)$\Rightarrow$(a) Since we always have the isomorphism 
$$G(V,\omega,\gamma)\cong G(V,\omega,\gamma^{-1}),(z,t,s)\rightarrow (z,t,-s),$$ we may assume that $\iota_G$ is a morphism of Lie groups. The Fock representation $\pi_F$ of $G(\hat V, \hat\omega, \hat\gamma)$ is semibounded 
by  Proposition~\ref{prop:semibfockrep}, so that 
we obtain the semibounded representation $\iota_G^*\pi_F$ 
of $G(V,\omega,\gamma)$. 
\end{prf}

\end{document}